\documentclass[11pt]{amsart}

\usepackage{amsmath,amssymb,amscd,amsfonts,verbatim}
\usepackage{paralist}
\usepackage[mathscr]{eucal}
\usepackage{hyperref}
\usepackage{color}

\newtheorem{thm}{Theorem}[section]
\newtheorem{lem}[thm]{Lemma}

\newtheorem{prop}[thm]{Proposition}
\newtheorem{cor}[thm]{Corollary}

\theoremstyle{definition}
\newtheorem{ex}[thm]{Example}
\newtheorem{remark[thm]}{Remark}
\newtheorem{assu-nota}[thm]{Assumption--Notation}
\newtheorem{step}{Step}

\newcommand{\inv}{^{-1}}

\newcommand{\Z}{\mathbb Z}

\newcommand{\OO}{\mathcal O}
\newcommand{\ct}{\chi_{\rm top}}
\DeclareMathOperator{\Alb}{Alb}

\DeclareMathOperator{\Pic}{Pic}

\DeclareMathOperator{\pr}{pr}

\newcommand{\epsi}{\varepsilon}

\newcommand{\si}{\sigma}

\newcommand{\dalb}{\delta_{\rm alb}}

\renewcommand{\>}{\!\!>}

\numberwithin{equation}{section}

\title[A uniform bound on the canonical degree \dots]{A uniform bound on the canonical degree of Albanese defective curves on surfaces}
\author{Margarida Mendes Lopes and Rita Pardini}

\thanks{
The  first author is a member of the Center for Mathematical
Analysis, Geometry and Dynamical Systems (IST/UTL); the second author is a member of G.N.S.A.G.A. of In.D.A.M.. This research was partially supported by  FCT (Portugal) through program POCTI/FEDER and Project 
PTDC/MAT/099275/2008 and by italian MIUR through PRIN 2008 project ``Geometria delle variet\`a algebriche e dei loro spazi di moduli". }

\begin{document}

\begin{abstract} Let $S$ be a minimal complex surface of general type with irregularity $q\ge 2$ and let $C\subset S$ be an irreducible curve of geometric genus $g$. Assume that $C$ is {\em Albanese defective}, i.e., that the image of $C$ via the Albanese map does not generate the Albanese variety $\Alb(S)$; we obtain a linear upper bound in terms of $K^2_S$ and $g$ for the canonical degree $K_SC$ of $C$. As a corollary,  we obtain a bound for the canonical degree of curves with $g\le q-1$, thereby generalizing and sharpening the main result of \cite{lu}.
\medskip


\medskip

\end{abstract}

\maketitle

\section{Introduction }

Let $S$ be an irregular minimal surface of general type with irregularity $q\ge2$ and  let $a\colon S\to\Alb(S)$ be  the Albanese map. Let $C\subset S$ be a curve; we define the {\em Albanese defect}   $\dalb(C)$ to be the codimension in $\Alb(S)$  of $<\!\!a(C)\!\!>$, the smallest abelian subvariety of $\Alb(S)$ containing $a(C)$. If $\dalb(C)>0$, then we say that $C$ is {\em Albanese defective}.

  In this note we give  a  uniform linear bound  for  the  canonical degree $K_SC$ of  an Albanese defective curve $C$  in terms of $K_S^2$ and of the geometric genus $g$  of $C$.  For a general discussion on  bounds of the canonical degree of curves see \cite{lu}, \cite{lumi} and \cite{miorbi}.
  
We prove the following theorem:

\begin{thm}\label{main} Let $S$ be a minimal surface of general type with irregularity $q\geq 2$ and let $C$ be an irreducible curve of $S$ with geometric genus $g$ such that  $\dalb(C)>0$.  Then:
\begin{enumerate}
\item if $\dalb(C)\ge 2$, then $K_SC\leq K_S^2-\dalb(C)$;

\item  if $\dalb(C)=1$, then  $K_SC\leq 4K_S^2+2g-2$ and there is a fibration $f\colon S\to B$, such that   $g(B)=1$  and $C$ is a component of a fibre.
\end{enumerate}
\end{thm}
If $C$ is an Albanese defective curve, 
 then the map $S\to   A/\!\!\<a(C)\>$ is a non constant morphism that contracts $C$ to a point,  hence  one has $C^2\le 0$. In \cite{lumi} it is proven that if $S$ is a  (not necessarily irregular) surface of general type and $C\subset S$   is an irreducible curve of genus $g$  with   $C^2\le 0$ then  $K_SC\le 6c_2(S)-2K^2_S+6g-6$. So   Theorem \ref{main} can be viewed   as a refinement of this inequality in the special case of an Albanese defective curve. Note also that the bounds  in Theorem  \ref{main} imply  $p_a(C)-g(C)\le 2K^2_S$, namely  for an Albanese defective curve the difference between the arithmetic genus and the geometric genus is bounded by a fixed quantity depending only on the surface $S$. Explicit examples of Albanese defective curves with several singularities are not easy to produce; a construction can be found in Example \ref{ex:sym2}

If a curve $C$ has geometric genus $g<q$, then $\dalb(C)\ge q-g>0$. So we have the following immediate consequences of Theorem \ref{main}:
\begin{thm}\label{thm:lu} Let $S$ be a minimal surface of general type with irregularity $q\geq 2$ and let $C$ be an irreducible curve of $S$ with geometric genus $g<q$.  Then:
\begin{enumerate}
\item if $g\leq q-2$, then $K_SC\leq K_S^2-(q-g)$;

\item  if $g=q-1$, then either $K_SC\leq K_S^2-2$, or  there is a fibration $f\colon S\to B$, where  $g(B)=1$, and $C$ is a component of a fibre. In this case $K_SC\leq 4K_S^2+2q-4$.
\end{enumerate}
\end{thm}

\begin{cor}\label{cor:classes}
Let $S$ be a  surface of general type with irregularity $q\geq 2$. Then:
\begin{enumerate}
\item the cohomology classes in $H^2(S,\Z)$ of  the irreducible curves $C$ of geometric genus $<q$ are a finite set;
\item if $S$ has no nontrivial morphism onto a curve of genus $>0$, then $S$ contains finitely many curves of geometric  genus $<q$.
\end{enumerate}
\end{cor}

Theorem  \ref{thm:lu} and Corollary \ref{cor:classes}, (ii)  generalize a result of S.Y. Lu (\cite{lu})  for curves of geometric genus $\leq 1$.

The first  part of the  proof  of Theorem \ref{main}  uses the same ideas as  the first part of the proof in  \cite{lu}. However  the remainder of the proof is   different   and shorter, and it  does not require a deep analysis of the behaviour of ramification  of the Albanese map.
This simplification is made possible by the   use of  two extra ingredients: the smoothness of bicanonical curves (see, e.g., \cite{ciro}) and the Severi inequality for minimal irregular surfaces  of maximal Albanese dimension (\cite{rita}).

\medskip     
If $F$ is a smooth fiber of a fibration $S\to E$, with $E$ an elliptic curve, then $K_SF=2g(F)-2$; since there
 are surfaces with infinitely many   fibrations onto an elliptic curve  with $g(F)\to\infty$  (see \cite{msri}, Example 2.1.2), any bound for the canonical degree of  Albanese defective curves must be $\ge 2g-2$,  and so  the term $2g-2$ in Theorem \ref{main} (ii) seems to be optimal. However, the bound is possibly not sharp;  if it were,  then several
     inequalities used in the proof should be  equalities, all at the same time.
     In particular the Severi inequality should be an equality. Conjecturally
     this happens only for $q=2$,  but it has been proven only under the assumption that the canonical class of the surface be ample (\cite{manetti}, cf. also \cite{severi-bis}); in the general case it has been an open question for many years.
 Also,  the coefficient  $4$ in the  term in function of $K_S^2$ seems too large;  still, we find it intriguing that  the different proof for the case of elliptic curves in \cite{lu}   also gives the same coefficient $4$.

\medskip In  \S \ref{sec:remarks} we compute $K_SC$ for some  examples of Albanese defective curves $C$.
\medskip

\noindent{\bf Notation and conventions:}  
 All varieties are complex projective; a {\em surface} is a smooth projective surface. The standard notation for the invariants of a surface $S$ is used:  $q(S):=h^0(\Omega^1_S)=h^1(\OO_S)$ is the {\em irregularity},  $p_g(S):=h^0(K_S)$ is the {\em geometric genus} and $\chi(\OO_S)=1+p_g(S)-q(S)$ is the (holomorphic) Euler-Poincar\'e characteristic and  $c_1(S)$, $c_2(S)$ are the  Chern classes of the tangent bundle of   $S$. 
$\ct(X)$ denotes the topological Euler characteristic of a  space $X$. 
 
 Recall  that $K_S^2=c_1(S)^2$,  $c_2(S)=\ct(S)$  and that  the Noether formula $K_S^2+c_2(S)=12\chi(\OO_S)$ holds. 
 
 An effective divisor $D$  on $S$ is  {\em $k$-connected} if for every decomposition $D=A+B$, with $A,B>0$ one has $AB\ge k$. We sometimes refer to effective divisors as ``curves''. As usual we denote by $p_a(D)$ the arithmetic genus of a curve and by $g(D)$ the geometric genus of an irreducible curve. 
 
 The Jacobian of a smooth curve $B$ is denoted by $J(B)$.

\section{Preliminaries}
In this section we collect several facts that will be of use in the proof of Theorem \ref{main}.

We recall  the following corollary of Arakelov's theorem.
\begin{prop}[see \cite{appendix}, Corollaire] \label{fibration} Let $f\colon S\to B$ be  a fibration of a smooth minimal surface  $S$ onto  a smooth curve $B$ of genus $b$, with general fibre of  genus $h\geq 2$.  Then 
$$ 8(h-1)(b-1)\leq K_S^2.$$
\end{prop} 

  Next we give a lower bound  for the canonical degree  of some curves of   $S$.
  \begin{lem}\label{lem:lowerbound} Let $S$ be a smooth minimal surface of general type  and let $D$ be a  curve such that $K_S-D>0$. 
  Then $$K_SD\ge q- \dalb(D).$$
  \end{lem}
   \begin{proof} By the $2$-connectedness of canonical divisors we have $D(K_S-D)\ge 2$, namely $K_SD\ge D^2+2$. If $D$ is irreducible,  by the adjunction formula we have:
   $$2K_SD\ge 2p_a(D)\ge  2g(D)\ge 2(q-\dalb(D)).$$
   In general,  the inequality follows by observing that if $D_1,\dots D_r$ are the distinct irreducile components  of $D$, then $K_SD\ge \sum_iK_SD_i$ and  $q-\dalb(D)\le \sum_i(q-\dalb(D_i))$.
     \end{proof}
  
 We will need also the following: 
  
\begin{prop}\label{fibre}   Let $S$ be a  surface and  let $f\colon S\to B$ be a relatively minimal fibration  with general fibre   $F$ of genus $g \geq 1$. Let $F_0$ be a singular fibre of $f$ and $C$ a component of $F_0$ appearing with multiplicity one.  If  $N_C$ is the normalization of $C$ then  $$\ct (F_0)-\ct(F)\geq 
 \dfrac {1}{2}K_SC+ \chi(\OO _{N_C}).$$

Furthermore if $F_0-C>0$  then the inequality is strict.
\end{prop}

\begin{proof}

Write $F_0=C+G+R$, where $C+G$ is the support of $F_0$ and $R$ is an effective (possibly 0) divisor with support contained in $G$. 

\smallskip
First we look at  the term  $-\ct(F)$.  Recall that
 $$-\ct(F)=-2\chi(\OO_F)= K_SF+F^2,$$  where the last equality is  the adjunction formula. 

On the other hand, since $F_0=C+G+R$,   $$K_SF+F^2=K_S(C+G)+(C+G)^2+ 2(C+G)R+K_SR+R^2.$$

By the numerical properties of fibres (see, e.g., \cite{be2}), $RF=0$, i.e., 

\noindent   $R^2+R(C+G)=0$. Since we are assuming that $f\colon S\to B$ is a relatively minimal fibration, $K_SR\geq 0$ and so we obtain 

$$ -\ct(F)=K_SF+F^2\geq  K_SC+C^2+K_SG+G^2+ 2CG+ (C+G)R.$$

\bigskip

Now we  turn to  $\ct( F_0)=\ct (C+G)$.  By the additivity of the topological characteristic, one has $ \ct (C+G)=  \ct (C)+ \ct (G)- d$, where $d$ is the number of intersection points of $C$ and $G$. Note that $d\leq CG$.

\medskip By  the proof  of  Lemme VI.5 of \cite{be2}  given a reduced curve $D$, one has $$\ct(D)\geq \chi(\OO_D)+\chi(\OO_{N_D}),$$ where $N_D$ is the normalization of $D$.
In particular,  for any reduced curve $D$ one has  $\ct(D)\geq  2 \chi(\OO_D)$.

So $ \ct (C)+ \ct (G)- d\geq  \chi(\OO_C)+\chi(\OO_{N_C})+2\chi(\OO_G)-d$, where $N_C$ is the normalization of $C$.

This can be written, by the adjunction formula,  $$ \ct (C)+ \ct(G)- d\geq  -\dfrac {1}{2}(K_SC+C^2)+\chi(\OO_ {N_C})-K_SG-G^2-d.$$

Then $\ct(F_0)-\ct(F)\geq    -\dfrac {1}{2}(K_SC+C^2)+\chi(\OO_{N_C})-K_SG-G^2-d+  K_SC+C^2+K_SG+G^2+ 2CG+ (C+G)R= $

$=\dfrac {1}{2}(K_SC+C^2)+ \chi(\OO _{N_C})+2CG+(C+G)R-d$.

Now 
$ \dfrac {1}{2}(K_SC+C^2)+ \chi(\OO_ {N_C})+2CG+(C+G)R-d=$

$= \dfrac {1}{2}K_SC+ \chi(\OO _{N_C})+  \dfrac {1}{2} (C^2+ CG+CR)+ (CG-d)+ \dfrac {1}{2}  G(C+R)+ \dfrac 12 R(C+G).$

\medskip

 From $CF=0$ we obtain $C^2+CG+CR=0$.  On the other hand  the fibre $F_0$ is not a multiple fibre, because $C$ appears  with multiplicity 1 on $F_0$. Thus $F_0$  is $1$-connected and so, if  $G+R\neq 0$, then   $\dfrac {1}{2}( G(C+R)+ R(C+G))>0$. Since also $CG-d\geq 0$ we obtain 

 $$\ct (F_0)-\ct (F)\geq  \dfrac {1}{2}K_SC+ \chi(\OO _{N_C})$$
 
 and that if $C\neq F_0$, then the inequality is strict.

 \end{proof}

\section{The proof of Theorem \ref{main}}\label{sec:proof}

We denote  by   $a\colon S\to \Alb(S)$ the Albanese map, we set $$A_C:=\Alb(S)/\!\!\<a(C)\>$$ and we denote by $a_C\colon S\to A_C$ the induced map. Observe that $A_C$ is an abelian variety of dimension $\dalb(C)$ and that $a_C$ is a non constant morphism that contracts $C$ to a point.

\begin{step}  {\em If $a_C(S)$ is a surface, then $K_SC\le K_S^2-\dalb(C)$.}\smallskip\\
 Let $\alpha, \beta$ be the pull back on $S$ of two general $1$-forms of $A_C$ and let $G$ be the divisor of zeros of $\alpha\wedge \beta$. Then $G=C+R$ with $R$ is an effective divisor whose image via $a_C$  spans $A_C$, hence $\dalb(R)\le q-\dalb(C)$. So  we have $$K^2_S=K_SG= K_SC+K_SR\ge K_SC+ \dalb(C),$$ where the last inequality  follows by  Lemma \ref{lem:lowerbound}.
\end{step}

So from now we may assume that $a_C(S)$ is a curve $B$.

\begin{step}{\em   $B$  is smooth of genus $\dalb(C)$ and $a_C$ has connected fibers.}\label{st:genusB}\smallskip\\
The map  $a_C$ factors through an irrational pencil $f_C\colon S \to \Gamma$, where $\Gamma$ is a smooth curve of genus $\gamma\ge \dalb(C)$. On the other hand, $H^0(\omega_{\Gamma})$ pulls back to a $\gamma$-dimensional subspace of $H^0(\Omega^1_S)$ whose elements  restrict to $0$ on $C$, so $\gamma \le \dalb(C)$;   therefore $\gamma =\dalb(C)$ and the natural map $J(\Gamma)\to A_C$ is an isogeny. By the universal property of the Albanese variety, the  map $\Alb(S)\to J(\Gamma)$ induced by $f_C$ descends to a map $A_C\to J(\Gamma)$ such that $A_C\to J(\Gamma)\to A_C$ is the identity, hence $J(\Gamma)\to A_C$ is an isomorphism and $B$ is smooth of genus $\dalb(C)$.
\end{step}
 
 \smallskip
 
\begin{step} {\em If $\dalb(C)>1$, then $4(\dalb(C)-1)K_SC\le K^2_S$.}\label{st:d>1} \smallskip\\ We denote by $F$ a general fiber of $f_C$, which is a smooth curve of genus $g(F)\ge 2$ since $S$ is of general type.
 By the adjunction formula $K_SF=2g(F)-2$. Since $K_S$ is nef,   $K_SC\leq K_SF$ and so,
by Proposition  \ref{fibration} and Step \ref{st:genusB},   we obtain $$4(\dalb(C)-1)K_SC\leq 8(g(F)-1)(\dalb(C)-1)\leq K_S^2.$$
\end{step}

\smallskip

\begin{step} {\em If $\dalb(C)=1$, then $K_SC\le 4K^2_S+2g-2$.}\smallskip\\
In this case $A_C$ is an elliptic curve and by Step \ref{st:genusB} the map $a_C$ has connected fibers; denote by   $F_0$ the fiber of $a_C$ which contains  $C$. 

 If $F_0$ is non singular then $K_SF_0=K_SC=2g-2$ and the inequality is trivially satisfied.
 
 So assume $F_0$ is singular.

Suppose that $C$ appears in $F_0$  with multiplicity $\geq 2$. Since $|2K_S|$ is base point free (see, e.g., \cite{ciro}),   we can consider     a smooth curve $D$  in $|2K_S|$ meeting $C$ transversally.  Let $m:=K_SF_0$ and let $D\to A_C$ be the cover of degree $2m$ induced by $a_C$.  Then one has $6K_S^2= 2g(D)-2$ by the adjunction formula and,  by  the Hurwitz formula,  $2g(D)-2= \deg R$, where $R$ is the ramification divisor of $D\to A_C$.  Since $C$ appears with multiplicity $\geq 2$ in $F_0$,  we obtain $\deg R\geq 2K_SC$.  So  $6K^2_S\geq  2K_SC$,  i.e. $3K_S^2\geq K_SC$. 

Suppose now that $C$ appears in $F_0$  with multiplicity $1$ and let $F$ denote a general fibre of $f$.  By  Proposition \ref{fibre},  $$\ct(F_0)-\ct(F)\geq 
\dfrac{1}{2}(K_SC)+1-g.$$

Then, by the formulas for $c_2$ of a fibered  surface (see, e.g., Lemme VI.4 of   \cite{be2}), we have
 $c_2(S)\geq\ct (F_0)-\ct(F)$. 
 
 Notice that the Albanese image of $S$ is a surface, since $q\ge 2$ by assumption and on a surface with Albanese dimension 1 the Albanese defect of any curve is either equal to $0$ or to $q$.  Hence by the Severi inequality (\cite{rita})  we have $ c_2(S)\leq 2K_S^2$. So one obtains 
 $\dfrac{1}{2}(K_SC)+1-g \leq 2K_S^2$, i.e. $ K_SC\leq 4K_S^2+2g-2$.
\end{step}
\smallskip

\begin{step}{\em end of proof.}\smallskip\\
We only have to show that in the situation of Step \ref{st:d>1} one has $K_SC\le K^2_S-\dalb(C)$. 

Assume for contradiction that $K_SC\ge K^2_S-(\dalb(C)-1)$; then by Step \ref{st:d>1} we have $K_SC\ge(\dalb(C)-1)(4K_SC-1)\ge 4K_SC-1$ and thus $K_SC=0$. Then $C$ is a $-2$ curve, $\dalb(C)=q$ and $K_S^2\le q-1$. The last equality  contradicts the fact that $p_g(S)\ge q$ since $S$ is of general type and $K^2_S\ge 2p_g(S)$ since  $S$ is also irregular. 

\end{step}
 
 This concludes the proof of Theorem \ref{main}. 
\section{Examples}\label{sec:remarks}

In this section we compute explicitly $K_SC$ for some Albanese defective curves $C$ contained in an irregular surface $S$. We examine products, symmetric covers and some double covers  of abelian surfaces, namely most of the standard examples of irregular surfaces of general type. (One can of course also produce surfaces containing Albanese defective
curves by taking suitable complete intersections in an abelian variety, but
it is not clear how to compute explicitly $K_SC$ and $K_S^2$ for such
examples).

We use the notation introduced in the previous sections. 

\begin{ex}[Curves in the product of two curves]
Let $D$ be a curve of genus $g\ge 2$, let $S:=D\times D$ and let $C\subset S$ be the diagonal. Here $A_C=J(D)$ and the quotient map $a_C\colon \Alb(S)=J(D)\times J(D)\to J(D)$ is defined by $(x,y)\mapsto x-y$, hence  in this case the image of $a_C$ is a surface. We have $\dalb(C)=g$, $K_SC=4(g-1)=K_S^2/2$. The map $S\to A_C$ maps $S$ onto a surface. The situation is the same for graphs of automorphisms and similar for graphs of maps between curves.

If $D_i$, $i=1,2$, are curves of genus $g_i\ge 2$,  $f\colon D_1\to D_2$ is a non constant morphism  and $C\subset S:= D_1\times D_2$ is the graph of $f$, then we also have $K_SC\le K^2_S/2$, with equality holding iff $f$ is \'etale. 
\end{ex}

\begin{ex}[Examples of curves with $\dalb(C)=1$ and $K_SC=K^2_S$]\label{ex:sharp}

 Let $E$ be an elliptic curve, set $A:=E\times E$ and let $\Delta\subset A$ be the diagonal. Fix  $x_1\ne x_2\in E$,  set $B:=\{x_1\}\times E+E\times\{x_1\}+\{x_2\}\times E+E\times\{x_2\}$, $P_{ij}=(x_i,x_j)\in A $, $i, j=1,2$.   Choose $L\in \Pic(A)$ such that $2L\cong B$ and $L|_{\Delta}\not \equiv P_{11}+P_{22}$, and let $\pi \colon X\to  A$ be the double cover associated to the relation $2L\equiv B$. The surface $X$ has 4 singular points of type $A_1$, occurring over the points $P_{ij}$. The minimal resolution $S\to X$ is obtained by pulling back $\pi$ to the blow up $\epsi\colon \hat{A}\to A$ of $A$ at the points $P_{ij}$ and normalizing.  If we denote by $E_{ij} \subset \hat{A}$ the exceptional curve over $P_{ij}$, by $B'$ the strict transform of $B$ and by $L'$ the line bundle $\epsi^*L(-\sum E_{ij})$, then
$p\colon S\to \hat{A}$ is the double cover associated to the relation $2L'\equiv B'$. Hence the surface $S$ is minimal of general type with $K^2_S=4$, $p_g(S)=q(S)=2$. Let $\Delta'\subset \hat{A}$ be the strict transform of $\Delta$ and let $C=p\inv(\Delta')$. The double cover $C\to \Delta'$ is \'etale, given by the relation $2L'|_{\Delta'}\equiv 0$; since $L'|_{\Delta'}$ is non trivial, $C$ is a smooth elliptic curve. Standard computations give $K_SC=4=K^2_S$. 

Now  pick $\eta\in \Pic^0(S)$ of finite order $m$ such that $\eta|_C$ has also order $m$, let $S'\to S$ be the \'etale cover given by $\eta$ and let $C'\subset S'$ be the preimage of $C$. Then $C'$ is an elliptic curve with $C'K_S'=4m=K_S^2$.
\end{ex}

\begin{ex}[Curves in the symmetric square of a curve, I]\label{ex:sym1}
Let $D$ be a curve of genus $q\ge 3$ and set $S:=S^2D$. The surface $S$ is minimal of general type with invariants $p_g(S)=q(q-1)/2$, $q(S)=q$, $K^2_S=(q-1)(4q-9)$.  Let $\si$ be an involution of $D$; the curve $C:=\{x=\si(x) | x\in D\}\subset S$ is a smooth curve  isomorphic to  $D/\si$. The pull back of $C$ on $D\times D$ is the graph of $\si$, hence it has self intersection $2-2q$; it follows $C^2=1-q$ and $K_SC=q-1+2g-2$, where $g$ is the genus of $C$. By the Hurwitz formula we have $K_SC\le 2q-2$ with equality holding iff $\si$ is base point free. The abelian variety $A_C$ is isomorphic to the generalized Prym of the double cover $D\to D/\si$  and the map $a_C\colon S\to A_C$ is induced by the Abel-Prym map. So,  unless $q=3$ and $g=2$,  $\dalb(C)=g-q>1$ and the image of $a_C$ is a surface. \end{ex}

\begin{ex}[Curves in the symmetric square of a curve, II]\label{ex:sym2}
Let $D$, $S$ and $\si$ be as in Example \ref{ex:sym2} and assume in addition that $E:=D/\si$ is an elliptic curve. Let $p\colon D\to E$ be the quotient map and let $f\colon S^2D\to E$ be defined by $(x,y)\mapsto x+y$. Let $c\in E$ and write $F_c:=f\inv(c)$; the pull  back $\tilde{F}_c$ of $F_c$ to $D\times D$  is the preimage of $E_c:=\{z+w=c\}\subset E\times E$ via the map $p\times p\colon D\times D\to E\times E$. Denote by $B$ the branch locus of $p$ and by $\pr_i\colon E\times E\to E$ the projection onto the $i$-th factor, $i=1,2$. The map $\tilde{F}_c\to E_c$ is a $\Z_2^2$-cover, which is the fiber product of two double covers branched on $\pr_i^*B$. By the Hurwitz formula $B$ consists of $2q-2$ distinct points and for $c\in E$ general the curve  $\tilde{F}_c$ is smooth of  genus $4q-3$. It is easy to check that the general  $\tilde{F}_c$ meets the diagonal $\Delta_D\subset D\times D$ transversely at $8$ points and therefore, again by the Hurwitz formula, it is smooth of genus $g:=2q-1$. 

Write $B=x_1+\dots +x_{2q-2}$. If, say, $x_1=-x_2$ and $2x_1\ne 0$, then $\tilde{F}_0$ has a node over the point $(x_1, -x_1)$ and a node  over the point $(-x_1,x_1)$; these nodes are identified by the involution that exchanges the factors of $D\times D$, hence $F_0$ has a node. (If $x_1=-x_1$, then $\tilde{F}_0$ has also a node over $(x_1,-x_1)=(x_1,x_1)$, but $F_0$ is smooth at the corresponding point).
In a similar way one can arrange that $F_0$ has any number $r\le q-1$ of nodes. For $r\le q-2$ the curve $\tilde{F}_0$ stays irreducible,  hence $F_0$ is also irreducible. (For $r=q-1$ it is possible to show that $F_0$ splits as the union of two smooth curves  $C_1$, $C_2$ with $\dalb(C_i)\ge 2$ for $i=1,2$ and $q\ge 4$.)

\end{ex}

\bigskip

\begin{minipage}{13.0cm}
\parbox[t]{6.5cm}{Margarida Mendes Lopes\\
Departamento de  Matem\'atica\\
Instituto Superior T\'ecnico\\
Universidade T{\'e}cnica de Lisboa\\
Av.~Rovisco Pais\\
1049-001 Lisboa, PORTUGAL\\
mmlopes@math.ist.utl.pt
 } \hfill
\parbox[t]{5.5cm}{Rita Pardini\\
Dipartimento di Matematica\\
Universit\`a di Pisa\\
Largo B. Pontecorvo, 5\\
56127 Pisa, Italy\\
pardini@dm.unipi.it}
\end{minipage}

\end{document}